\documentclass{llncs}

\usepackage{amssymb,stmaryrd,mathrsfs,dsfont,mathtools,enumitem}
\usepackage[ruled,linesnumbered]{algorithm2e}
\usepackage{algpseudocode}

\input xy
\xyoption{all}

\usepackage{tikz}
\tikzstyle{vertex}=[circle, draw, inner sep=0pt, minimum size=2.7pt]
\newcommand{\vertex}{\node[vertex]}

\begin{document}
\frontmatter          
\pagestyle{headings}  

\mainmatter              

\title{Characterization of Double-Arborescences and their Minimum-Word-Representants}

\titlerunning{Title}  
%

\author{Tithi Dwary \and 
K. V. Krishna} 

\authorrunning{Tithi Dwary \and K. V. Krishna} 

\institute{Indian Institute of Technology Guwahati\\
	\email{tithi.dwary@iitg.ac.in};\;\;\; 
	\email{kvk@iitg.ac.in}}

\maketitle              

\begin{abstract}
A double-arborescence is a treelike comparability graph with an all-adjacent vertex. In this paper, we first give a forbidden induced subgraph characterization of double-arborescences, where we prove that double-arborescences are precisely $P_4$-free treelike comparability graphs. Then, we characterize a more general class consisting of $P_4$-free distance-hereditary graphs using split-decomposition trees. Consequently, using split-decomposition trees, we characterize double-arborescences and one of its subclasses, viz., arborescences; a double-arborescence is an arborescence if its all-adjacent vertex is a source or a sink. In the context of word-representable graphs, it is an open problem to find the classes of word-representable graphs whose minimum-word-representants are of length $2n - k$, where $n$ is the number of vertices of the graph and $k$ is its clique number. Contributing to the open problem, we devise an algorithmic procedure and show that the class of double-arborescences is one such class. It seems the class of double-arborescences is the first example satisfying the criteria given in the open problem, for an arbitrary $k$.
\end{abstract}


\section{Introduction }

This work aims at characterizing a special class of comparability graphs, viz., double-arborescences, and also finding their minimum-word-representants. The graphs in this work are simple and connected. In this section and Section \ref{split-dec-trees}, we present the requisite background material and fix the notation. 

A graph $G = (V, E)$ is called a comparability graph if it admits a transitive orientation, i.e., an assignment of direction to the edges of $G$ such that $\overrightarrow{ab} \in E$ and $\overrightarrow{bc} \in E$, then $\overrightarrow{ac} \in E$. Based on a given transitive orientation, every comparability graph $G$ induces a partially ordered set (in short, poset), denoted by $P_G$.  If $G$ admits a transitive orientation such that the Hasse diagram of the poset $P_G$ is a tree, then $G$ is called a treelike comparability graph and the corresponding orientation is called a treelike orientation. It was shown in \cite{cornelsen2009treelike} that a treelike orientation of a comparability graph, if exists, is unique up to isomorphism and reversing the whole orientation. A treelike comparability graph $G$ with an all-adjacent vertex is called a double-arborescence, i.e., there is a vertex $r$ such that $V = \{r\} \cup N_G(r)$, where $N_G(r)$ is the neighborhood of $r$ in $G$. We consider $r$ as the root of $G$. In addition, under the treelike orientation, if the root $r$ is a source (or a sink), i.e., the indegree (respectively, outdegree) of $r$ is zero, then $G$ is called an arborescence. The treelike orientation of a (double-)arborescence is called the (double-)arborescence orientation. Note that every arborescence is a double-arborescence but not conversely. We call a double-arborescence as a strict-double-arborescence if it is not an arborescence. A treelike comparability graph $G$ is called a path of $k$ double-arborescences (or simply, a path of double-arborescences) if $G$ (precisely) consists of $k$ number of double-arborescences and a path connecting their roots. The smallest possible $k$ such that $G$ is a path of $k$ double-arborescences can be determined in linear time \cite{cornelsen2009treelike}. These graphs are well depicted and understood by their Hasse diagrams or the digraphs of transitive reductions\footnote{The transitive reduction of a transitive orientation is obtained by deleting the transitive edges.}.  For example, refer Fig. \ref{fig_4} for these types of graphs given by their transitive reductions, in which (b) is a strict-double-arborescence. For details on the arborescences and related graphs, one may refer to \cite{golumbic2022they} and the references thereof.

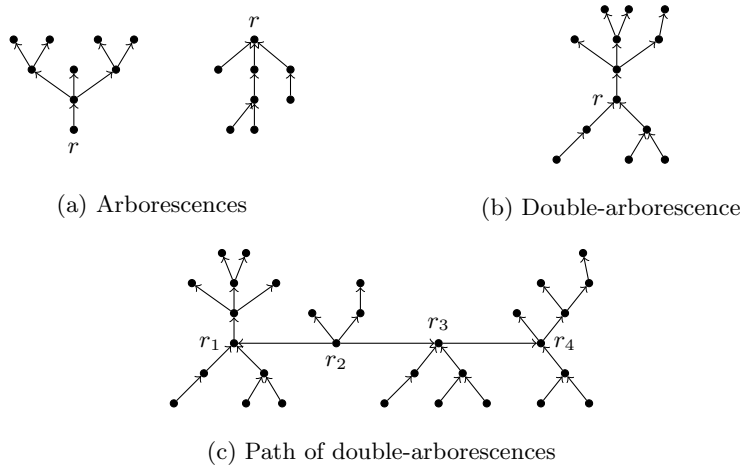
\begin{figure}[t]
	\centering
	\begin{minipage}{.5\textwidth}
		\centering
		\[\begin{tikzpicture}[scale=0.8]
			\vertex (a_1) at (0,0) [fill=black] {};
			\vertex (a_2) at (-0.7,0.5) [fill=black] {};
			\vertex (a_3) at (0.7,0.5) [fill=black] {};
			
			\vertex (a_4) at (-1,1) [fill=black] {};
			\vertex (a_5) at (-0.4,1) [fill=black] {};
			\vertex (a_6) at (0.4,1) [fill=black] {};
			\vertex (a_7) at (1,1) [fill=black] {};
			\vertex (a_8) at (0,-0.5) [fill=black,label=below:$r$] {};
			\vertex (a_9) at (0,0.5) [fill=black] {};
			
			\path[->]
			(a_8) edge (a_1)
			(a_1) edge (a_2)
			(a_1) edge (a_3)
			(a_2) edge (a_4)
			(a_2) edge (a_5)
			(a_3) edge (a_6)
			(a_3) edge (a_7)
			(a_1) edge (a_9);

			\vertex (a'_1) at (3,1) [fill=black,label=above:$r$] {};
			\vertex (a'_2) at (2.4,0.5) [fill=black] {};
			\vertex (a'_3) at (3.6,0.5) [fill=black] {};
			\vertex (a'_4) at (3,0.5) [fill=black] {};
			\vertex (a'_5) at (3,0) [fill=black] {};
			\vertex (a'_6) at (3.6,0) [fill=black] {};
			\vertex (a'_7) at (3,-0.5) [fill=black] {};
			\vertex (a'_8) at (2.6,-0.5) [fill=black] {};

			\path[<-]
			
			(a'_1) edge (a'_2)
			(a'_1) edge (a'_3)
			(a'_1) edge (a'_4)
			(a'_4) edge (a'_5)
			(a'_3) edge (a'_6)
			(a'_5) edge (a'_7)
			(a'_5) edge (a'_8);

		\end{tikzpicture}\] 
		(a) Arborescences
	\end{minipage}%
	\begin{minipage}{.5\textwidth}
		\centering
		
		\[\begin{tikzpicture}[scale=0.8]
			
			\vertex (a_1) at (0,1) [fill=black] {};
			\vertex (a_2) at (-0.7,1.5) [fill=black] {};
			\vertex (a_3) at (0.7,1.5) [fill=black] {};
			
			\vertex (a_6) at (0.8,2) [fill=black] {};
			
			\vertex (a_8) at (0,0.5) [fill=black,label=left:$r$] {};
			\vertex (a_9) at (0,1.5) [fill=black] {};
			\vertex (a_{10}) at (-0.2,2) [fill=black] {};
			\vertex (a_{11}) at (0.2,2) [fill=black] {};
			\vertex (a'_2) at (-0.5,0) [fill=black] {};
			\vertex (a'_3) at (0.5,0) [fill=black] {};
			\vertex (a'_5) at (-1,-0.5) [fill=black] {};
			\vertex (a'_6) at (0.2,-0.5) [fill=black] {};
			\vertex (a'_7) at (0.8,-0.5) [fill=black] {};
			
			\path[->]
			(a_8) edge (a_1)
			(a_1) edge (a_2)
			(a_1) edge (a_3)
			
			(a_1) edge (a_9)
			(a_9) edge (a_{10})
			(a_9) edge (a_{11})
			(a_3) edge (a_6);
			
			\path[->]
			
			(a'_2) edge (a_8)
			(a'_3) edge (a_8)
			
			(a'_5) edge (a'_2)
			(a'_6) edge (a'_3)
			(a'_7) edge (a'_3);

		\end{tikzpicture}\]	
		(b) Double-arborescence
	\end{minipage}%
	
	\begin{minipage}{.5\textwidth}
		\centering
		
		\[\begin{tikzpicture}[scale=0.8]
			
			\vertex (a_1) at (0,1) [fill=black] {};
			\vertex (a_2) at (-0.7,1.5) [fill=black] {};
			\vertex (a_3) at (0.7,1.5) [fill=black] {};

			\vertex (a_8) at (0,0.5) [fill=black,label=left:$r_1$] {};
			\vertex (a_9) at (0,1.5) [fill=black] {};
			\vertex (a_{10}) at (-0.2,2) [fill=black] {};
			\vertex (a_{11}) at (0.2,2) [fill=black] {};
			\vertex (a'_2) at (-0.5,0) [fill=black] {};
			\vertex (a'_3) at (0.5,0) [fill=black] {};
			\vertex (a'_5) at (-1,-0.5) [fill=black] {};
			\vertex (a'_6) at (0.2,-0.5) [fill=black] {};
			\vertex (a'_7) at (0.8,-0.5) [fill=black] {};
			\vertex (1) at (1.7,0.5) [fill=black,label=below:$r_2$] {};
			\vertex (2) at (1.3,1) [fill=black] {};
			\vertex (3) at (2.1,1) [fill=black] {};
			\vertex (4) at (2.1,1.5) [fill=black] {};
			\vertex (5) at (3.4,0.5) [fill=black,label=above:$r_3$] {};
			\vertex (6) at (3,0) [fill=black] {};
			\vertex (7) at (3.8,0) [fill=black] {};
			\vertex (8) at (3.4,-0.5) [fill=black] {};
			\vertex (9) at (4.2,-0.5) [fill=black] {};
			\vertex (10) at (2.5,-0.5) [fill=black] {};
			\vertex (11) at (5.1,0.5) [fill=black,label=right:$r_4$] {};
			\vertex (12) at (4.7,1) [fill=black] {};
			\vertex (13) at (5.5,1) [fill=black] {};
			\vertex (14) at (5.1,1.5) [fill=black] {};
			\vertex (15) at (5.9,1.5) [fill=black] {};
			\vertex (16) at (5.5,0) [fill=black] {};
			\vertex (17) at (5.1,-0.5) [fill=black] {};
			\vertex (18) at (5.9,-0.5) [fill=black] {};
			\vertex (19) at (5.8,2) [fill=black] {};

			\path[->]
			(a_8) edge (a_1)
			(a_1) edge (a_2)
			(a_1) edge (a_3)
			
			(a_1) edge (a_9)
			(a_9) edge (a_{10})
			(a_9) edge (a_{11})
			(1) edge (2)
			(1) edge (3)
			(3) edge (4)
			(6) edge (5)
			(7) edge (5)
			(8) edge (7)
			(9) edge (7)
			(10) edge (6)
			(11) edge (12)
			(11) edge (13)
			(13) edge (14)
			(13) edge (15)
			(16) edge (11)
			(17) edge (16)
			(18) edge (16)
			(15) edge (19);
			
			\path[->]
			(1) edge (a_8)
			(1) edge (5)
			(5) edge (11);
			
			\path[->]
			
			(a'_2) edge (a_8)
			(a'_3) edge (a_8)
			
			(a'_5) edge (a'_2)
			(a'_6) edge (a'_3)
			(a'_7) edge (a'_3);

		\end{tikzpicture}\]	
		(c) Path of double-arborescences
	\end{minipage}%
	\caption{Examples of treelike comparability graphs in terms of transitive reductions} \label{fig_4}
\end{figure}

The arborescences were first studied by Wolk in \cite{wolk1962comparability,wolk1965note} and characterized them as $(C_4, P_4)$-free graphs, i.e., the graphs in which none of $P_4$ and $C_4$ is present as an induced subgraph, where $C_n$ is a cycle on $n$ vertices and $P_n$ is a path on $n$ vertices. Further, Golumbic characterized the arborescences as trivially perfect graphs, i.e., the graphs in which for every induced subgraph the size of a largest independent set equals the number of maximal cliques \cite{golumbic1978trivially}. Based on the aforementioned characterizations, linear-time algorithms for recognizing the arborescences are presented in the literature (see \cite{yan1996quasi,chu2008simple}). The arborescences were further generalized and studied in \cite{jung_1978}. In \cite{cornelsen2009treelike}, Cornelsen and Di Stefano proved that a graph is a path of double-arborescences if and only if it is a treelike permutation graph. However, no characterization is available for double-arborescences in the literature.

A word over a finite set of letters is a finite sequence written by juxtaposing the letters of the sequence. A subword $u$ of a word $w$ is a subsequence of the sequence $w$, denoted by $u \ll w$. For instance, $abcabb \ll abbcacbba$. Let $w$ be a word over a set $A$, and $B$ be a subset of $A$. We write $w_B$ to denote the subword of $w$ that precisely consists of all occurrences of the letters of $B$. For example, if $w=abbcacbba$, then $w_{\{b, c\}} = bbccbb$. We say that the letters $a$ and $b$ alternate in $w$ if $w_{\{a, b\}}$ is of the form either $ababa\cdots$ or $babab\cdots$, which can be of even or odd length. A word $w$ is called $k$-uniform if every letter occurs exactly $k$ times in $w$. A graph $G = (V, E)$ is called a word-representable graph if there is a word $w$ with the symbols of $V$ such that, for all $a, b \in V$, $a$ and $b$ are adjacent in $G$ if and only if $a$ and $b$ alternate in $w$; such a word $w$ is called a word-representant of $G$. If a graph is word-representable, then it has infinitely many word-representants \cite{MR2467435}. A word-representable graph $G$ is said to be $k$-word-representable if a $k$-uniform word represents it. It is known that every word-representable graph is $k$-word-representable, for some $k$ \cite{MR2467435}. For a comprehensive introduction to the topic of word-representable graphs, one may refer to the monograph \cite{words&graphs} by Kitaev and Lozin. Further, a minimum-word-representant of a word-representable graph $G$ is a shortest (in terms of its length) word-representant of $G$. The length of a minimum-word-representant of $G$ is denoted by $\ell(G)$. Note that a minimum-word-representant of a word-representable graph need not be uniform. Let $G$ be a word-representable graph on $n$ vertices. It is evident that $\ell(G) \ge n$. It is known that the class of circle graphs\footnote{A circle graph is an intersection graph of chords of a circle.} characterizes the 2-word-representable graphs \cite{MR2914710}. Thus, an obvious upper bound for $\ell(G)$ of a circle graph is $2n$. In the seminal work \cite{Marisa_2020}, Gaetz and Ji considered the subclasses, viz., cycles and trees, of circle graphs and provided explicit formulae for both the length and the number of minimum-word-representants. In \cite{Eshwar_2024}, Srinivasan and Hariharasubramanian proved that there is no circle graph $G$  with $\ell(G) = 2n$ and an edgeless graph $G$ is the only circle graph having $\ell(G) = 2n-1$. Moreover, they showed that $\ell(G) = 2n-2$ for a triangle-free circle graph $G$ containing at least one edge. However, they established through an example (see \cite[Example 2.15]{Eshwar_2024}) that $\ell(G)$ need not be $2n-k$ for a  word-representable graph $G$ with clique number\footnote{A clique is a complete subgraph. The size of a maximum clique in a graph is its clique number.} $k$. In this connection, they posed an open problem to find classes of word-representable graphs $G$ with clique number $k$ such that $\ell(G) = 2n-k$. So far, no examples of such graph classes are available  in the literature for an arbitrary $k$.

In this work, we employ the notion of split-decomposition trees and  characterize double-arborescences as well as arborescences. Further, we find the minimum-word-representants of double-arborescences and show that this class of graphs serves as an example for the above-mentioned open problem from \cite{Eshwar_2024}. In Section 2, we recall the notion of split-decomposition trees and present relevant results from the literature.  In Section 3, we first provide a forbidden induced subgraph characterization of double-arborescences, where we prove that double-arborescences are precisely $P_4$-free treelike comparability graphs. Also, using split decomposition we obtain an alternative proof for the well-known characterization of arborescences given in \cite[Theorem 3]{wolk1965note} (also see \cite{wolk1962comparability}). Next, we consider a more general class of treelike comparability graphs, viz., distance-hereditary graphs. We introduce a notion called s-leaf-path in the minimal split-decomposition tree of a distance-hereditary graph, and using this notion we characterize the class of $P_4$-free distance-hereditary graphs. Consequently, we obtain characterizations of arborescences and double-arborescences with respect to their minimal split-decomposition trees. Finally, in Section 4, we note that arborescences and double-arborescences are word-representable graphs and devise an algorithm based on breadth-first search to construct minimum-word-representants of  arborescences. Moreover, using the algorithm, we also obtain minimum-word-representants of double-arborescences. We prove that if $G$ is a double-arborescence on $n$ vertices with clique number $k$, then $\ell(G) = 2n - k$.

\section{Split-Decomposition Trees}
\label{split-dec-trees}

In this section, we recall the concepts of split decomposition of a connected graph (from \cite{circlegraph3}) and graph-labelled trees (from \cite{graph-labelled_2012}), and present them in a unified framework for fixing the notation of split-decomposition trees. 

The concept of split decomposition was introduced by Cunningham in \cite{cunningham_2,cunningham_1} and it was used to recognize certain classes of graphs such as circle graphs \cite{circlegraph3}, parity graphs \cite{cicerone1999extension}, and distance-hereditary graphs \cite{DHgraph1}. Recently, split decomposition is also used in the context of word-representable graphs \cite{tithi}. Let $G = (V, E)$ be a connected graph.  A split of $G$ is a partition $\{V_1, V_2\}$ of $V$ such that each of $V_1$ and $V_2$ contains at least two vertices, and every vertex in $N_G(V_1)$ is adjacent to every vertex in $N_G(V_2)$, where, for $A \subseteq V$,  $N_G(A) = \bigcup_{a \in A} N_G(a) \setminus A$, called the neighborhood of $A$. A prime graph is a graph without any split. 

A split decomposition of a graph $G = (V, E)$ with split $\{V_1, V_2\}$  is represented as a disjoint union of the induced subgraphs $G[V_1]$ and $G[V_2]$ along with an edge $e = \overline{v_1v_2}$, where $v_1$ and $v_2$ are two new vertices such that $v_1$ and $v_2$ are adjacent to each vertex of $N_G(V_2)$  and $N_G(V_1)$, respectively. We call $v_1$ and $v_2$ as marked vertices and $e$ as a marked edge. By deleting the edge $e$, we obtain two components with vertex sets $V_1 \cup \{v_1\}$ and $V_2 \cup \{v_2\}$ called the split components. The two components are then decomposed recursively to obtain a split decomposition of $G$. A minimal split decomposition of a graph is a split decomposition whose split components can be cliques, stars\footnote{A star is a tree on $n$ vertices with one vertex, called the center, of degree $n-1$.} and prime graphs such that the number of split components is minimized. While there can be multiple split decompositions of a graph, a minimal split decomposition of a graph is unique \cite{cunningham_2,cunningham_1}. 

\begin{figure}[t]
	\centering
	
	\[\begin{tikzpicture}[scale=0.8]
		\vertex (a_1) at (-1, 1) [fill=black,label=above:$1'$] {};
		\vertex (a_2) at (-1, -1) [fill=black,label=below:$2'$] {};
		\vertex (a_3) at (0,0) [fill=black,label=above:${\alpha_1}$] {};
		\vertex (a_4) at (1,0) [fill=black,label=below:${\alpha_2}$] {};
		\vertex (a_5) at (1,1) [fill=black,label=above:$3'$] {};
		\vertex (a_6) at (2,0) [fill=black,label=above:${\alpha_3}$] {};
		\vertex (a_7) at (2,-1) [fill=black,label=below:$4'$] {};
		\vertex (a_8) at (3,0) [fill=black,label=below:${\alpha_4}$] {};
		\vertex (a_9) at (4,1) [fill=black,label=above:$5'$] {};
		\vertex (a_{10}) at (4,-1) [fill=black,label=below:$6'$] {};
		\path
		(a_1) edge (a_3)
		(a_2) edge (a_3)
		(a_3) edge (a_4)
		(a_4) edge (a_5)
		(a_4) edge (a_6)
		(a_6) edge (a_7)
		(a_6) edge (a_8)
		(a_8) edge (a_9) 
		(a_8) edge (a_{10});	
	\end{tikzpicture}\] 
	\caption{A tree $T$} \label{Tree}
\end{figure}
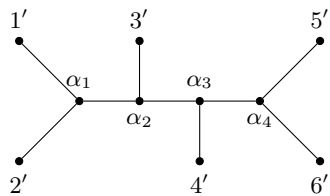 

We now recall the concept of graph-labelled tree introduced by Gioan and Paul \cite{graph-labelled_2012} and its relation with a split decomposition of a graph. Let $T$ be a tree and $\mathcal{F}$ be a family of vertex disjoint graphs. A graph-labelled tree, denoted by $T_{\mathcal{F}}$, is the graph obtained from $T$ by labelling (in fact, by inserting) every internal vertex $\alpha$ of degree $k$ ($\ge 2$) by a graph $G_{\alpha} \in \mathcal{F}$ on $k$ vertices such that there is a bijection from the edges of $T$ incident to $\alpha$ to the vertices of $G_{\alpha}$ (and by replacing the end point $\alpha$ of a tree edge with the corresponding vertex of $G_\alpha$). Note that the set of pendant (degree one) vertices of $T_{\mathcal{F}}$ is precisely the set of leaves of $T$. Given a graph-labelled tree $T_{\mathcal{F}}$ and the family $\mathcal{F}$, we can determine the underlying tree structure $T$. For clarity, we encircle the graphs $G_\alpha$ in $T_{\mathcal{F}}$ that are replacing the internal vertices $\alpha$ of $T$. For instance, for the tree given in Fig. \ref{Tree}, a graph-labelled tree is depicted in Fig. \ref{fig_3}(b). The edges of $T_\mathcal{F}$ in the encircled portions are called $\mathcal{F}$-edges and the remaining edges are called $T$-edges, as they correspond to the tree edges of $T$. A path in $T_{\mathcal{F}}$ is called an alternated path if it alternates between $T$-edges and $\mathcal{F}$-edges. A maximal alternated path is an alternated path that cannot be extended by adding more edges while maintaining the alternatedness. The accessibility graph of a graph-labelled tree $T_{\mathcal{F}}$ is the graph, denoted by $T^A_{\mathcal{F}}$, in which the vertex set is the set of pendant vertices in $T_{\mathcal{F}}$ and any two vertices $a$ and $b$ in $T^A_{\mathcal{F}}$ are adjacent if and only if there is an alternated path between $a$ and $b$ in $T_{\mathcal{F}}$.

Let $\mathcal{H}$ be a split decomposition of a graph $G$. Extend $\mathcal{H}$ by adding one new pendant vertex $a'$ for each non-marked vertex $a$ of $\mathcal{H}$ such that $a$ and $a'$ are adjacent. The graph thus extended can be viewed as a graph-labelled tree $T_{\mathcal{F}}$, called a split-decomposition tree of $G$, where $\mathcal{F}$ consists of the split components of $\mathcal{H}$. If the split components of a split-decomposition tree are cliques (called clique components) and stars (called star components), then it is called a clique-star tree.  We rewrite the uniqeness result by Cunningham in terms of graph-labelled trees in the following theorem and call such a graph-labelled tree as a minimal split-decomposition tree. For example, refer Fig. \ref{fig_3} for a graph and its minimal split-decomposition tree.

\begin{theorem}[\cite{MR3907778,cunningham_2,graph-labelled_2012}]\label{reduced_cs_tree}
	For every connected graph $G$, there exists a unique split-decomposition tree $T_{\mathcal{F}}$ of $G$ such that
	\begin{enumerate}[label=\rm (\roman*)]
		\item\label{acc_iso} the accessibility graph $T^A_{\mathcal{F}}$ is isomorphic to the graph $G$,
		\item every split component is a clique, a star, or a prime graph,
		\item\label{point3} every split component has at least three vertices, 
		\item there is no marked edge with end points in two clique components, and
		\item there is no marked edge between the center of a star component and a leaf of another star component. 
	\end{enumerate}
\end{theorem}

Treelike comparability graphs were characterized using split decomposition in \cite[Theorem 4]{cornelsen2009treelike}. In view of the algorithm provided in \cite[Theorem 5]{cornelsen2009treelike}, we rewrite the characterization of treelike comparability graphs in the setting of graph-labelled trees.

 \begin{theorem}[\cite{cornelsen2009treelike}]\label{Treelike_ch}
	Let $G$ be a graph and $T_{\mathcal{F}}$ be the minimal split-decomposition tree of $G$. Then $G$ is a treelike comparability graph if and only if 
	\begin{enumerate}[label=\rm (\roman*)]
		\item $T_{\mathcal{F}}$ is a clique-star tree,
		\item each clique component has at most two marked vertices, and
		\item there is no marked edge between the centers of two star components. 
	\end{enumerate}
\end{theorem}

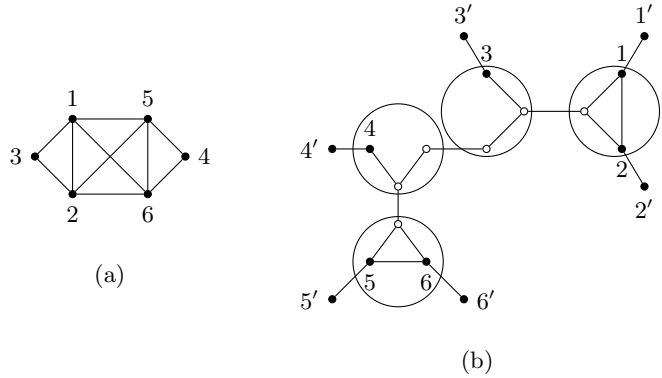
\begin{figure}[t]
	\centering
	\begin{minipage}{.3\textwidth}
		\centering
		\[\begin{tikzpicture}
			\vertex (a_1) at (-1, 0) [fill=black,label=left:$3$] {};
			\vertex (a_2) at (-0.5, 0.5) [fill=black,label=above:$1$] {};
			\vertex (a_3) at (-0.5,-0.5) [fill=black,label=below:$2$] {};
			\vertex (a_4) at (0.5,0.5) [fill=black,label=above:$5$] {};
			\vertex (a_5) at (0.5,-0.5) [fill=black,label=below:$6$] {};
			\vertex (a_6) at (1,0) [fill=black,label=right:$4$] {};
			
			\path
			(a_1) edge (a_2)
			(a_1) edge (a_3)
			(a_2) edge (a_3)
			(a_2) edge (a_4)
			(a_2) edge (a_5)
			(a_3) edge (a_5)
			(a_3) edge (a_4)
			(a_4) edge (a_5) 
			(a_4) edge (a_6)
			(a_5) edge (a_6);
			
		\end{tikzpicture}\] 
			(a)
	\end{minipage}%
	\begin{minipage}{.5\textwidth}
		\centering
		
		\[\begin{tikzpicture}[scale=1]
			
			\vertex (a_1) at (-1, 0.5) [fill=black, label=above:$3$] {};
			\vertex (a_2) at (0.8, 0.5) [fill=black, label=above:$1$] {};
			\vertex (a_3) at (0.8, -0.5) [fill=black, label=below:$2$] {};
			\vertex (a_4) at (-2.55, -2) [fill=black, label=below:$5$] {};
			\vertex (a_5) at (-1.8, -2) [fill=black, label=below:$6$] {};
			\vertex (a_6) at (-2.55, -0.5) [fill=black, label=above:$4$] {};
			\vertex (a_7) at (-2.175, -1) [label=left:$ $] {};
			\vertex (a_8) at (-2.175, -1.5) [label=left:$ $] {};
			\vertex (a_9) at (-1.8, -0.5) [label=above:$ $] {};
			\vertex (a_{10}) at (-1, -0.5) [label=below:$ $] {};
			\vertex (a_{11}) at (-0.5, 0) [label=left:$ $] {};
			\vertex (a_{12}) at (0.3,0) [label=below:$ $] {};
			\vertex (b_1) at (-1.3, 1) [fill=black,label=above:$3'$] {};
			\vertex (b_2) at (1.1, 1) [fill=black,label=above:$1'$] {};
			\vertex (b_3) at (1.1, -1) [fill=black,label=below:$2'$] {};
			\vertex (b_6) at (-3.05, -0.5) [fill=black,label=left:$4'$] {};
			\vertex (b_4) at (-3.05, -2.5) [fill=black,label=left:$5'$] {};
			\vertex (b_5) at (-1.3, -2.5) [fill=black,label=right:$6'$] {};
			
			\path
			(a_1) edge (a_{11})
			(a_2) edge (a_3)
			(a_2) edge (a_{12})
			(a_3) edge (a_{12})
			(a_4) edge (a_5)
			(a_4) edge (a_8)
			(a_5) edge (a_8)
			(a_6) edge (a_7)
			(a_1) edge (b_1)
			(a_7) edge (a_9)
			(a_2) edge (b_2)
			(a_{10}) edge (a_{11})
			(a_3) edge (b_3)
			(a_6) edge (b_6)
			(a_4) edge (b_4)
			(a_5) edge (b_5)
			(a_{12}) edge (a_{11})
			(a_{10}) edge (a_9)
			(a_7) edge (a_8);
			\draw[] (0.7,0) circle[radius=0.6cm];
			\draw[] (-1,0) circle[radius=0.6cm];
			\draw[] (-2.175,-0.5) circle[radius=0.6cm];
			\draw[] (-2.175,-2) circle[radius=0.6cm];
		\end{tikzpicture}\]	
		      (b)
	\end{minipage}%
	
\caption{(a) A graph, and (b) its minimal split-decomposition tree} \label{fig_3}
\end{figure}

A graph $G$ is called a distance-hereditary graph if the distance between any two vertices in any connected induced subgraph of $G$ is same as the distance in $G$. The class of distance-hereditary graphs is more general and includes treelike comparability graphs \cite{cornelsen2009treelike}. In \cite{DH_graph}, a multiple characterizations of distance-hereditary graphs were obtained based on various parameters. The distance-hereditary graphs were characterized as the graphs whose minimal split-decomposition trees are clique-star trees \cite{Bouchet_1}. Further, certain subclasses of distance-hereditary graphs were characterized in terms of their minimal split-decomposition trees in \cite{MR3907778}. In particular, $C_4$-free distance-hereditary graphs were characterized as per the following result. 

\begin{theorem}[\cite{MR3907778}] \label{lemma_C4}
	Let $G$ be a distance-hereditary graph with the minimal split-decomposition tree $T_{\mathcal{F}}$. Then $G$ is $C_4$-free if and only if $T_{\mathcal{F}}$ does not have any center-center path (i.e., an alternated path whose endpoints are the centers of star components, and it does not contain any edge from either of star components). 
\end{theorem}

In this work, we characterize the class of $P_4$-free distance-hereditary graphs in terms of their minimal split-decomposition trees (see Theorem \ref{lemma_P4}). In this connection, we need the following properties of maximal alternated paths in split-decomposition trees.

\begin{lemma}[\cite{MR3907778}] \label{lemma_1}
	Let $T_{\mathcal{F}}$ be the minimal split-decomposition tree of a distance-hereditary graph $G$ and $G_{\alpha}$ be a split component of $T_{\mathcal{F}}$. We have the following properties of maximal alternated paths in $T_{\mathcal{F}}$.
	\begin{enumerate}[label=\rm (\roman*)]
		\item\label{maximal_alter_path} There exists a maximal alternated path from any vertex of $G_{\alpha}$ but does not contain any edge of $G_{\alpha}$.
		
		\item Any maximal alternated path starting from a vertex of $G_{\alpha}$ ends in a pendant vertex of $T_{\mathcal{F}}$.
		
		\item\label{point_3} Let $P$ and $Q$ be two maximal alternated paths from distinct vertices of $G_{\alpha}$. If $P$ and $Q$ do not contain any edge of $G_{\alpha}$, then they end at distinct pendant vertices of $T_{\mathcal{F}}$.
	\end{enumerate}	
\end{lemma}

\section{Characterizations}

In this section, we aim to characterize double-arborescences and arborescences in terms of their split-decomposition trees. For these subclasses, first we provide forbidden induced subgraph characterizations. Further, we obtain a necessary and sufficient condition to identify a more general class, viz., $P_4$-free distance-hereditary graphs. Consequently, we give one more characterization each for double-arborescences and arborescences.       

\begin{lemma}\label{path_double_arb}
	Suppose $G$ is a path of $k$ double-arborescences for some $k \ge 2$ such that $k$ is the smallest possible. Then $G$ contains $P_4$ as an induced subgraph.
\end{lemma}

\begin{proof}
	Let $D$ be the digraph corresponding to the treelike orientation of $G$. Further, let $D'$ be the transitive reduction of $D$, i.e., the spanning subgraph of $D$ obtained by deleting the transitive edges. Let $G_i$ with root $r_i$, $1 \le i \le k$, be the $k$ double-arborescences in $G$ such that $\langle r_1, r_2, \ldots, r_k \rangle$ is the root-path of $D'$.   
	
	Suppose the edge $\overline{r_1r_2}$ is oriented as $\overrightarrow{r_1r_2}$ in $D'$. 
	We claim that there exists a vertex $a_1$ in $G_1$ such that $\overrightarrow{r_1a_1}$ is in $D'$. On the contrary, suppose $\overrightarrow{ar_1}$ in $D$ for all vertices $a$ in $G_1$. Hence, we have $\overrightarrow{ar_2}$ in $D$ for all vertices $a$ in $G_1$. Then, $G$ is a path of $k - 1$ double-arborescences, viz., $G_1 \cup G_2, G_3, \ldots, G_k$ with the root-path $\langle r_2, r_3, \ldots, r_k \rangle$ in $D'$; a contradiction to the minimality of $k$. Further, there exist a vertex $a_2$ in $\displaystyle\bigcup_{2 \le i \le k} G_i$ such that $\overrightarrow{a_2r_t}$ is in $D'$, for some $t \ge 2$,  as shown in the following cases:   
	\begin{itemize}
		\item Case-1: $\overrightarrow{r_{i+1}r_i}$ in $D'$ for some $i \ge 2$. Let $t$ be the least such that $\overrightarrow{r_{t+1}r_t}$ in $D'$. Then we choose $a_2$ to be $r_{t+1}$.
		
		\item Case-2: $\overrightarrow{r_{i}r_{i+1}}$ for all $i < k$ in $D'$. For each $i \ge 2$, if $\overrightarrow{r_ia}$ in $D$ for all vertices $a$ in $G_i$,  then $\overrightarrow{r_1a}$ in $D$ for all vertices $a$ in $\displaystyle\bigcup_{2 \le i \le k} G_i$. In which case, $G$ is a double-arborescence with the root $r_1$; a contradiction to the minimality of $k$. Hence, there exists $a_2$ in $G_t$, for some $t \ge 2$, such that $\overrightarrow{a_2r_t}$ is in $D'$.
	\end{itemize}
	The path $\langle a_1, r_1, r_t, a_2 \rangle$ is an induced $P_4$ in $G$. Indeed, other than these three edges, there will not be any more edges between $a_1, r_1, r_t$ and $a_2$ in $G$ as there is no directed path between $a_1$ and $r_t$; $a_1$ and $a_2$; or $r_1$ and $a_2$ in $D'$. 
	
	Similarly, one can observe that if $\overline{r_1r_2}$ is oriented as $\overrightarrow{r_2r_1}$ in $D'$, then there exist vertices $a_1$ in $G_1$ and $a_2$ in $\displaystyle\bigcup_{2 \le i \le k} G_i$ such that $\overrightarrow{a_1r_1}$ and $\overrightarrow{r_ta_2}$ are in $D'$ for some $t \ge 2$ so that $\{a_1, r_1, r_t, a_2\}$ induces a $P_4$ in $G$. \qed
\end{proof}

We now give a forbidden induced subgraph characterization for \break double-arborescences within treelike comparability graphs.

\begin{theorem}\label{P_4-free}
	A graph $G$ is a double-arborescence if and only if $G$ is a $P_4$-free treelike comparability graph. 
\end{theorem}

\begin{proof}
	Suppose $G$ is a double-arborescence with a root $r$. Since $G$ is a treelike comparability graph, by Theorem \ref{Treelike_ch}, $G$ is distance-hereditary. If $G$ contains a $P_4$ induced by, say, $\{a_1, a_2, a_3, a_4\}$, then clearly $r \neq a_i$, for all $1 \leq i \leq 4$, as $r$ is an all-adjacent vertex of $G$. Note that the graph induced by $\{r, a_1, a_2, a_3, a_4\}$ is the complement of $K_1 \cup P_4$ (see Gem in Fig. \ref{fig_8}). But, by \cite[Theorem 3.1]{Stefano_2012}, it is a forbidden induced subgraph for a distance-hereditary comparability graph. Hence, $G$ is $P_4$-free.
	
	Conversely, suppose a treelike comparability graph $G$ is $P_4$-free. Note that $P_4$-free graphs are permutation graphs (cf. \cite{Bose_1998}). Thus, by \cite[Theorem 6]{cornelsen2009treelike}, $G$ is a path of $k$ double-arborescences such that $k$ is the smallest number of double-arborescences in $G$. If $k \ge 2$, then by Lemma \ref{path_double_arb}, $G$ induces a $P_4$; which is not possible. Hence, $k = 1$ so that $G$ is a double-arborescence.
	\qed
\end{proof}

\begin{lemma}\label{prop_arbor}
	Suppose $G$ is a strict-double-arborescence with a root $r$. There exist two non-adjacent vertices $a_1$ and $a_2$ in $G$ such that  $\overrightarrow{ra_1}$ and $\overrightarrow{ra_2}$. Similarly, there exist two non-adjacent vertices $a_3$ and $a_4$ in $G$ such that  $\overrightarrow{a_3r}$ and $\overrightarrow{a_4r}$.
\end{lemma}	

\begin{proof}
	Suppose $G = (V, E)$ is a strict-double-arborescence with a root $r$. Let $D$ be the digraph corresponding to the treelike orientation of $G$ and $D'$ be the transitive reduction of $D$. Let  $A_1 = \{a \in V \mid \overrightarrow{ra} \ \text{in} \ D\}$ and $A_2 = \{a \in V \mid \overrightarrow{ar} \ \text{in} \ D\}$. Note that $A_1 \neq \varnothing$, $A_2 \neq \varnothing$, and $V = \{r\} \cup A_1 \cup A_2$.  Moreover, if $a, b \in A_1$ are adjacent in $G$, then they are on a directed path from $r$ in $D'$. Accordingly, if the vertices of $A_1$ form a clique in $G$, then the vertices of $A_1$ are on a directed path from $r$ in $D'$. In that case, $G$ can be seen as an arborescence whose root would be the end point, say $a'$, of the above mentioned directed path, i.e., $V = \{a'\} \cup \{a \in V \mid \overrightarrow{aa'} \ \text{in} \ D\}$; a contradiction to $G$ is a strict-double-arborescence. Hence, there exist two vertices $a_1, a_2 \in A_1$ such that $a_1$ and $a_2$ are not adjacent in $G$. Similarly, there exist two non-adjacent vertices in $A_2$. \qed
\end{proof}

\begin{theorem}\label{C_4-free}
	Let $G$ be a treelike comparability graph. Then $G$ is $C_4$-free if and only if $G$ does not contain a strict-double-arborescence as its induced subgraph.
\end{theorem}

\begin{proof}
	Let $D$ be the digraph corresponding to the treelike orientation of $G$. Suppose $G$ is $C_4$-free and contains a strict-double-arborescence, say $H$ with root $r$, as its induced subgraph. Then, by Lemma \ref{prop_arbor}, there exist two pairs of non-adjacent vertices $a_1$, $b_1$ and $a_2$, $b_2$ in $H$ such that $\overrightarrow{ra_1}$, $\overrightarrow{rb_1}$, $\overrightarrow{a_2r}$ and $\overrightarrow{b_2r}$ exist in $D$. Then $\{a_1, b_1, a_2, b_2\}$ induces a $C_4$, viz., $\langle a_1, a_2, b_1, b_2, a_1 \rangle$, in $H$ and hence in $G$; which is a contradiction.
	
	Conversely, suppose $G$ does not contain a strict-double-arborescence as its induced subgraph. If $G$ contains a $C_4$ induced by $\{a_1, b_1, a_2, b_2\}$, then by \cite[Proposition 4.2]{dobson2004treelike}, $N_G(a_1) \cap N_G(b_1) \cap N_G(a_2) \cap N_G(b_2)$ induces a complete subgraph of $G$. Let $a$ be a vertex in $N_G(a_1) \cap N_G(b_1) \cap N_G(a_2) \cap N_G(b_2)$, then observe that the graph induced by $\{a, a_1, b_1, a_2, b_2\}$ is a strict-double-arborescence as shown in Fig. \ref{Double-arborescence}; a contradiction. 
	\qed
\end{proof}

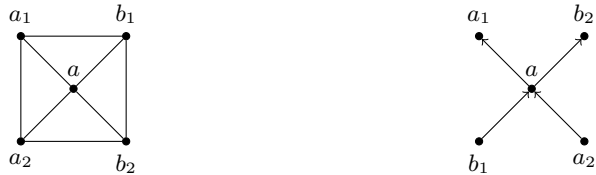
\begin{figure}[t]
	\centering
	\begin{minipage}{.5\textwidth}
		\centering
		\[\begin{tikzpicture}[scale=0.7]	
			\vertex (a_1) at (0,0) [fill=black,label=above:$a$] {};
			\vertex (a_2) at (1,1) [fill=black,label=above:$b_1$] {};
			\vertex (a_3) at (1,-1) [fill=black,label=below:$b_2$] {};
			\vertex (a_4) at (-1,1) [fill=black,label=above:$a_1$] {};
			\vertex (a_5) at (-1,-1) [fill=black,label=below:$a_2$] {};
			
			\path
			(a_1) edge (a_2)
			(a_1) edge (a_3)
			(a_1) edge (a_4)
			(a_1) edge (a_5)
			(a_2) edge (a_4)
			(a_4) edge (a_5)
			(a_3) edge (a_5)
			(a_2) edge (a_3);
		\end{tikzpicture}\] 
	\end{minipage}%
	\begin{minipage}{.5\textwidth}
		\centering	
		\[\begin{tikzpicture}[scale=0.7]
			\vertex (a_1) at (0,0) [fill=black,label=above:$a$] {};
			\vertex (a_2) at (1,1) [fill=black,label=above:$b_2$] {};
			\vertex (a_3) at (1,-1) [fill=black,label=below:$a_2$] {};
			\vertex (a_4) at (-1,1) [fill=black,label=above:$a_1$] {};
			\vertex (a_5) at (-1,-1) [fill=black,label=below:$b_1$] {};
			
			\path[->]
			(a_3) edge (a_1)
			(a_1) edge (a_4);
			\path[->]
			(a_1) edge (a_2)
			(a_5) edge (a_1);
		\end{tikzpicture}\]
	\end{minipage}%
	\caption{Strict-double-arborescence.} \label{Double-arborescence}
\end{figure}

The following corollary is evident from Theorem \ref{P_4-free} and Theorem \ref{C_4-free}.

\begin{corollary} \label{ch_arborescence}
	A graph $G$ is an arborescence if and only if $G$ is a $(C_4, P_4)$-free treelike comparability graph.
\end{corollary}

We now characterize the class of $P_4$-free distance-hereditary graphs through the notion of s-leaf-paths in their minimal split-decomposition trees. Let $G$ be a distance-hereditary graph with the minimal split-decomposition tree $T_{\mathcal{F}}$. We call a leaf of a star component in $T_{\mathcal{F}}$ as s-leaf. An s-leaf-path in $T_{\mathcal{F}}$ is an alternated path $P$ such that the endpoints of $P$ are s-leaves of two star components and $P$ does not contain any edge from either of star components (e.g., refer Fig. \ref{fig_5}).

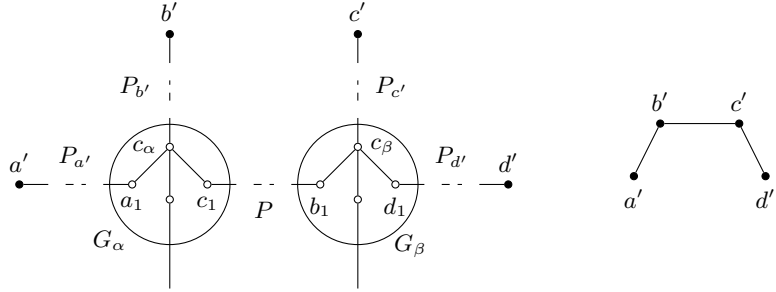
\begin{figure}[t]
	\centering
	\begin{minipage}{.65\textwidth}
		\centering
		\[\begin{tikzpicture}	
			\vertex (a_1) at (0,0) [fill=black,label=above:$a'$] {};
			\node (a_2) at (0.5,0) [] {};
			\node (a_3) at (1,0) [] {};
			\vertex (a_4) at (1.5, 0) [label=below:$a_1$] {};
			\vertex (a_5) at (2, 0.5) [label=left:$c_\alpha$] {};
			\vertex (a_6) at (2.5, 0) [label=below:$c_1$] {};
			\node (a_7) at (3, 0) [] {};
			\node (a_8) at (3.5, 0) [] {};
			\vertex (a_9) at (4, 0) [label=below:$b_1$] {};
			\vertex (a_{10}) at (4.5, 0.5) [label=right:$c_\beta$] {};
			\vertex (a_{11}) at (5, 0) [label=below:$d_1$] {};
			\node (a_{12}) at (5.5, 0) [] {};
			\node (a_{13}) at (6, 0) [] {};
			\vertex (a_{14}) at (6.5, 0) [fill=black,label=above:$d'$] {};
			\node (a_{15}) at (2, 1) [] {};
			\node (a_{16}) at (2, 1.5) [] {};
			\vertex (a_{17}) at (2, 2) [fill=black,label=above:$b'$] {};
			\node (a_{18}) at (4.5, 1) [] {};
			\node (a_{19}) at (4.5, 1.5) [] {};
			\vertex (a_{20}) at (4.5,2) [fill=black,label=above:$c'$] {};
			\vertex (a_{21}) at (2, -0.2) [] {};
			\node (a_{22}) at (2, -1.5) [] {};
			\vertex (a_{23}) at (4.5, -0.2) [] {};
			\node (a_{24}) at (4.5, -1.5) [] {};
			\node (a_{25}) at (1.2, -1.1) [label=above:$G_\alpha$] {};
			\node (a_{26}) at (5.2, -1.2) [label=above:$G_\beta$] {};
			\node (a_{27}) at (0.75,0) [label=above:$P_{a'}$] {};
			\node (a_{28}) at (3.25,0) [label=below:$P$] {};  
			\node (a_{29}) at (5.75,0) [label=above:$P_{d'}$] {};  
			\node (a_{30}) at (2,1.3) [label=left:$P_{b'}$] {};
			\node (a_{31}) at (4.5,1.3) [label=right:$P_{c'}$] {};

			\path
			(a_1) edge (a_2)
			(a_3) edge (a_4)
			(a_5) edge (a_4)
			(a_5) edge (a_6)
			(a_5) edge (a_{15})
			(a_6) edge (a_7)
			(a_8) edge (a_9)
			(a_{10}) edge (a_9)	
			(a_{10}) edge (a_{11})	
			(a_{10}) edge (a_{18})
			(a_{11}) edge (a_{12})
			(a_{13}) edge (a_{14})
			(a_{16}) edge (a_{17})
			(a_{19}) edge (a_{20})
			(a_{21}) edge (a_5)
			(a_{21}) edge (a_{22})
			(a_{10}) edge (a_{23})
			(a_{23}) edge (a_{24})	;
			
			\path [dashed]
			(a_2) edge (a_3)
			(a_7) edge (a_8)
			(a_{12}) edge (a_{13})
			(a_{15}) edge (a_{16})
			(a_{18}) edge (a_{19});
			
			\draw[] (2,0) circle[radius=0.8cm];
			\draw[] (4.5,0) circle[radius=0.8cm];	
		\end{tikzpicture}\] 	
	\end{minipage}%
	\begin{minipage}{.3\textwidth}
		\centering
		
		\[\begin{tikzpicture}[scale=0.7]
			
			\vertex (a_1) at (0,0) [fill=black,label=below:$a'$] {};
			\vertex (a_2) at (0.5,1) [fill=black,label=above:$b'$] {};
			\vertex (a_3) at (2,1) [fill=black,label=above:$c'$] {};
			\vertex (a_4) at (2.5,0) [fill=black,label=below:$d'$] {};
			
			\path
			(a_1) edge (a_2)
			(a_2) edge (a_3)
			(a_4) edge (a_3);
			
		\end{tikzpicture}\]	
	\end{minipage}%
	\caption{A portion of some $T_{\mathcal{F}}$ with an s-leaf-path $P$ and its accessibility graph  $P_4$} \label{fig_5}
\end{figure}

\begin{theorem} \label{lemma_P4}
	Let $G$ be a distance-hereditary graph and $T_{\mathcal{F}}$ be the minimal split-decomposition tree of $G$. Then $G$ contains an induced $P_4$ if and only if there exists an s-leaf-path in $T_{\mathcal{F}}$.
\end{theorem}

\begin{proof}
	 Since $G$ is isomorphic to the accessibility graph $T_{\mathcal{F}}^A$ (see Theorem \ref{reduced_cs_tree}\ref{acc_iso}), we prove that $T_\mathcal{F}^A$ contains an induced $P_4$ if and only if there exists an s-leaf-path in $T_{\mathcal{F}}$. Let $P$ be an s-leaf-path between two star components $G_\alpha$ and $G_\beta$ in $T_{\mathcal{F}}$. Let $c_1 \in G_\alpha$, $b_1 \in G_\beta$ be the endpoints of $P$. Since $T_{\mathcal{F}}$ is the minimal split-decomposition tree, both $G_\alpha$ and $G_\beta$ have at least three vertices each (see Theorem \ref{reduced_cs_tree}\ref{point3}). Thus, there are at least two s-leaves in each of the star components $G_\alpha$ and $G_\beta$. Then, by Lemma \ref{lemma_1}\ref{maximal_alter_path}, there are at least two maximal alternated paths from $G_\alpha$ that do not use any edge of $G_\alpha$ and, by Lemma \ref{lemma_1}\ref{point_3}, these maximal alternated paths end at distinct pendant vertices of $T_{\mathcal{F}}$. Of these paths, suppose one path, say $P_{a'}$, is from an s-leaf $a_1$ in $G_\alpha$ to a pendant vertex $a'$ of $T_{\mathcal{F}}$, and the other path, say $P_{b'}$, is from the center $c_\alpha$ of $G_\alpha$ to a pendant vertex $b'$ of $T_{\mathcal{F}}$. Similarly, there are at least two maximal alternated paths from $G_\beta$ one, say $P_{c'}$, is from its center $c_\beta$ to a pendant vertex $c'$ of $T_{\mathcal{F}}$, and the other, say $P_{d'}$, is from an s-leaf, say $d_1$, in $G_\beta$ to a pendant vertex $d'$ of $T_{\mathcal{F}}$, as shown in Fig. \ref{fig_5}.
	
	We show that $\langle a', b', c', d' \rangle$ is an induced path in $T_\mathcal{F}^A$. As shown in Fig. \ref{fig_5}, note that the path consisting of $P_{a'}$ followed by the edge $\overline{a_1c_\alpha}$ and then $P_{b'}$ is an alternated path from $a'$ to $b'$ in $T_{\mathcal{F}}$ so that $a'$ and $b'$ are adjacent in $T_\mathcal{F}^A$. Further, the path consisting of $P_{b'}$ followed by the edge $\overline{c_\alpha c_1}$, the path $P$, the edge $\overline{b_1c_\beta}$, then the path $P_{c'}$ is an alternated path from $b'$ to $c'$ in $T_{\mathcal{F}}$. Thus, $b'$ and $c'$ are adjacent in $T_\mathcal{F}^A$. Similarly, the vertices $c'$ and $d'$ are adjacent in $T_\mathcal{F}^A$.
	Consider the underlying tree $T$ of $T_{\mathcal{F}}$ and note that there is a unique path between $a'$ and $c'$ in $T$ that should pass through the vertices $\alpha$ and $\beta$ of $T$. Thus, any path between $a'$ and $c'$ in $T_{\mathcal{F}}$ should pass through $G_\alpha$ and $G_\beta$, and by construction, that path must use two edges of $G_\alpha$ so that it cannot be an alternated path. This shows that $a'$ and $c'$ are not adjacent in $T_\mathcal{F}^A$. Similarly, one can show that $\overline{b'd'}$ and $\overline{a'd'}$ are not in $T_\mathcal{F}^A$. Hence, $T_\mathcal{F}^A$ contains an induced $P_4$. 
	
	Conversely, suppose $T_\mathcal{F}^A$ has an induced $P_4$, say $\langle a', b', c', d'\rangle$. Since $\overline{b'a'}$, $\overline{b'c'}$ are edges of $T_\mathcal{F}^A$, there exist alternated paths, say $P_{b', a'}$ and $P_{b', c'}$, which begin at the pendant vertex $b'$ of $T_{\mathcal{F}}$ and end at the pendant vertices $a'$ and $c'$ of $T_{\mathcal{F}}$, respectively. Similarly, there exists an alternated path $P_{c', d'}$ between the pendant vertices $c'$ and $d'$ of $T_{\mathcal{F}}$. 
	
	Let $G_\alpha$ be the split component of $T_{\mathcal{F}}$ until which the paths $P_{b', a'}$ and $P_{b', c'}$ have the common segment and they split thereafter. Let $c_\alpha \in G_\alpha$ be the vertex at which the common segment ends, and $a_1$ and $c_1$ be the vertices of $G_\alpha$ at which the paths $P_{b', a'}$ and $P_{b', c'}$ exit $G_\alpha$, respectively. 
	We now show that $a_1$ and $c_1$ are not adjacent in $G_\alpha$ so that $G_\alpha$ is a star component. Otherwise, we can have an alternated path between $a'$ and $c'$, viz., $P_{a', a_1}$ followed by the edge $\overline{a_1c_1}$ of $G_\alpha$ and then $P_{c_1, c'}$, where $P_{a', a_1}$ is the segment of $P_{b', a'}$ between $a_1$ and $a'$, and  $P_{c_1, c'}$ is the segment of $P_{b', c'}$ between $c_1$ and $c'$. Thus, $a'$ and $c'$ are adjacent in $T_\mathcal{F}^A$; a contradiction. Hence, $G_\alpha$ must be a star component. It is evident that $c_\alpha$ is its center and $a_1$, $c_1$ are its s-leaves. 
	
	Let $G_\beta$ be the split component of $T_{\mathcal{F}}$ until which the paths $P_{c', d'}$ and $P_{b', c'}$ have the common segment and they split thereafter. Let $c_\beta \in G_\beta$ be the vertex at which the common segment ends, and $b_1$ and $d_1$ be the vertices of $G_\beta$ at which the paths $P_{b', c'}$ and $P_{c', d'}$ exit $G_\beta$, respectively. As shown above, observe that $G_\beta$ is a star component, in which $c_\beta$ is the center and $b_1, d_1$ are s-leaves.
	
	Note that the alternated path $P_{b', c'}$ passes through the split components $G_\alpha$ and $G_\beta$. We observe that the pendant vertex $b'$ is nearer to $G_\alpha$ than $G_\beta$ on the path $P_{b', c'}$. On the contrary, suppose $G_\beta$ is nearer to $b'$. Then, clearly, $G_\alpha$ would be nearer to $c'$. Since the path $P_{b', c'}$ exits $G_\alpha$ at $c_1$ towards the pendant vertex $c'$, the vertex $c_1$ is closest in $G_\alpha$ to $c'$. Similarly, the vertex $b_1$ is closest in $G_\beta$ to $b'$. Accordingly, the vertices $b', c', c_1, b_1, c_\alpha, c_\beta$ will appear in the following sequence on the path $P_{b', c'}$: $$b', b_1, c_\beta, c_\alpha, c_1, c'$$
	Thus, the segment of $P_{b', c'}$ between $c_\alpha$ and $c_\beta$ is a center-center path in $T_{\mathcal{F}}$, as shown in Fig. \ref{fig_6}. Then by Theorem \ref{lemma_C4}, $\langle a', b', c', d', a' \rangle$ forms a $C_4$ in $T_\mathcal{F}^A$; a contradiction to induced $P_4$ over these vertices. Hence, the pendant vertex $b'$ is nearer to $G_\alpha$ than $G_\beta$ on the path $P_{b', c'}$ so that the vertices $b', c', c_1, b_1, c_\alpha, c_\beta$ will appear in the following sequence on the path $P_{b', c'}$: $$b', c_\alpha, c_1, b_1, c_\beta, c'$$
	Evidently, the segment of $P_{b', c'}$ between $c_1$ and $b_1$ is an s-leaf-path between $G_\alpha$ and $G_\beta$ in $T_{\mathcal{F}}$. 
	\qed	
\end{proof}

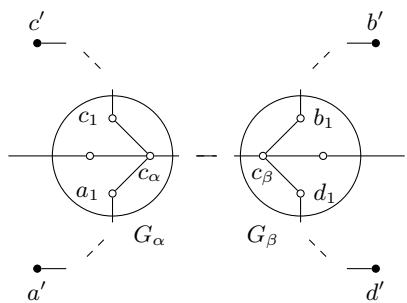
\begin{figure}[t]
	\centering
	\begin{minipage}{.5\textwidth}
		\centering
		\[\begin{tikzpicture}	
			\vertex (a_1) at (-1,1.5) [fill=black,label=above:$c'$] {};
			\node (a_2) at (-0.5,1.5) [] {};
			\node (a_3) at (0,1) [] {};
			\vertex (a_4) at (0, 0.5) [label=left:$c_1$] {};
			\vertex (a_5) at (0.5, 0) [label=below:$c_\alpha$] {};
			\vertex (a_6) at (0, -0.5) [label=left:$a_1$] {};
			\node (a_7) at (0, -1) [] {};
			\node (a_8) at (-0.5,-1.5) [] {};
			\vertex (a_9) at (-1,-1.5) [fill=black,label=below:$a'$] {};
			\node (a_{10}) at (1, 0) [] {};
			\node (a_{11}) at (1.5, 0) [] {};
			\vertex (a_{12}) at (2, 0) [label=below:$c_\beta$] {};
			\vertex (a_{13}) at (2.5, 0.5) [label=right:$b_1$] {};
			\node (a_{14}) at (2.5, 1) [] {};
			\node (a_{15}) at (3, 1.5) [] {};
			\vertex (a_{16}) at (3.5, 1.5) [fill=black,label=above:$b'$] {};
			\vertex (a_{17}) at (2.5, -0.5) [label=right:$d_1$] {};
			\node (a_{18}) at (2.5, -1) [] {};
			\node (a_{19}) at (3, -1.5) [] {};
			\vertex (a_{20}) at (3.5,-1.5) [fill=black,label=below:$d'$] {};
				\vertex (a_{21}) at (-0.3,0) [] {};
				\node (a_{22}) at (-1.5, 0) [] {};
				\vertex (a_{23}) at (2.8,0) [] {};
				\node (a_{24}) at (4, 0) [] {};
			\node (a_{25}) at (0.5, -0.7) [label=below:$G_\alpha$] {};
			\node (a_{26}) at (2, -0.7) [label=below:$G_\beta$] {};
			\path
			(a_1) edge (a_2)
			(a_3) edge (a_4)
			(a_5) edge (a_4)
			(a_5) edge (a_6)
			(a_5) edge (a_{10})
			(a_5) edge (a_{21})
			(a_6) edge (a_7)
			(a_8) edge (a_9)
				
			(a_{10}) edge (a_{11})	
			(a_{23}) edge (a_{12})
			(a_{11}) edge (a_{12})	
			(a_{13}) edge (a_{12})
			(a_{17}) edge (a_{12})
			(a_{13}) edge (a_{14})

			(a_{16}) edge (a_{15})
			(a_{18}) edge (a_{17})
			(a_{23}) edge (a_{24})
			(a_{19}) edge (a_{20})
			(a_{22}) edge (a_{21})	;
			
			\path [dashed]
			(a_2) edge (a_3)
			(a_7) edge (a_8)
			(a_{10}) edge (a_{11})
			(a_{15}) edge (a_{14})
			(a_{18}) edge (a_{19});
			
			\draw[] (0,0) circle[radius=0.8];
			\draw[] (2.5,0) circle[radius=0.8cm];
	
		\end{tikzpicture}\] 	
	\end{minipage}%
	\caption{A center-center path.} \label{fig_6}
\end{figure}

Using the split decomposition, we now present an alternative proof  for the well-known characterization of arborescences given in \cite[Theorem 3]{wolk1965note} (also see \cite{wolk1962comparability}).

\begin{theorem}\label{alter_proof}
	A graph $G$ is an arborescence if and only if $G$ is $(C_4, P_4)$-free.
\end{theorem}

\begin{proof}
	Suppose $G$ is an arborescence. Then, by Corollary \ref{ch_arborescence}, $G$ is a $(C_4, P_4)$-free graph.
	
	Conversely, suppose $G$ is a $(C_4, P_4)$-free graph. Using Theorem \ref{Treelike_ch}, we show that $G$ is a treelike comparability graph so that it is an arborescence (again by Corollary \ref{ch_arborescence}). Note that $G$ does not contain any of the graphs given in Fig. \ref{fig_8}, as each of them has $C_4$ or $P_4$ as an induced subgraph. Thus, by \cite[Theorem 3.1]{Stefano_2012}, $G$ is a distance-hereditary comparability graph so that the minimal split-decomposition tree $T_{\mathcal{F}}$ of $G$ is a clique-star tree. As $G$ is $(C_4, P_4)$-free, note that $T_{\mathcal{F}}$ has neither a center-center path (by Theorem \ref{lemma_C4}) nor an s-leaf-path (by Theorem \ref{lemma_P4}). In particular, there is no marked edge between the centers of two star components; otherwise, we will have a center-center path in $T_{\mathcal{F}}$. 
	
	Finally, we claim that each clique component of $T_{\mathcal{F}}$ has at most two marked vertices. On the contrary, suppose there is a clique component in $T_{\mathcal{F}}$ with three marked vertices say $a_1$, $a_2$ and $a_3$. For $1 \le i \le 3$, let $e_i$ be the marked edge in $T_{\mathcal{F}}$ with one end point $a_i$ and the other end point, say, $b_i$.  Since  $T_{\mathcal{F}}$ is the minimal split-decomposition tree of $G$, by statement (iv) of Theorem \ref{reduced_cs_tree}, each $b_i$ ($1 \le i \le 3$) is a vertex of a star component. Note that not all $b_i$'s can be centers of the respective star components; otherwise, $T_\mathcal{F}$ will have a center-center path, e.g., $\langle b_1, a_1, a_2, b_2\rangle$. Also, not all $b_i$'s can be s-leaves of the respective star components; otherwise, $T_\mathcal{F}$ will have an s-leaf-path, e.g., $\langle b_1, a_1, a_2, b_2\rangle$. Hence, one of $b_i$'s will be the center (or a s-leaf), without loss of generality assume that it is $b_1$, and the other two will be s-leaves (or the centers, respectively) of the respective star components. In any case, $\langle b_2, a_2, a_3, b_3\rangle$ is an s-leaf-path or center-center path in $T_\mathcal{F}$, which is not possible in $T_\mathcal{F}$. Thus, there will be at most two marked vertices in any clique component of $T_\mathcal{F}$. Hence, by Theorem \ref{Treelike_ch}, $G$ is a treelike comparability graph.  \qed 
	\end{proof}

\begin{figure}[t]
	\centering
	\begin{minipage}{.2\textwidth}
		\centering
		\[\begin{tikzpicture}[scale=0.6]
			
			\vertex (1) at (0,0) [fill=black] {};
			\vertex (2) at (0,1.1) [fill=black] {};
			\vertex (3) at (1.1,0) [fill=black] {};
			\vertex (4) at (1.1,1.1) [fill=black] {};
			\vertex (5) at (0.55,1.9) [fill=black] {};
			
			\path 
			(1) edge (2)
			(1) edge (3)
			(2) edge (4)
			(3) edge (4)
			(2) edge (5)
			(4) edge (5);

		\end{tikzpicture}\]	
		House
	\end{minipage}%
	\begin{minipage}{.2\textwidth}
		\centering
		
		\[\begin{tikzpicture}[scale=0.6]

			\vertex (1) at (0,0) [fill=black] {};
			\vertex (2) at (0.7,0) [fill=black] {};
			\vertex (3) at (1.4,-0.2) [fill=black] {};
			\vertex (4) at (-0.7,-0.2) [fill=black] {};
			\vertex (5) at (0.35,-1.6) [fill=black] {};
			
			\path 
			(4) edge (1)
			(1) edge (2)
			(2) edge (3)
			(5) edge (1)
			(5) edge (2)
			(5) edge (3)
			(5) edge (4);

		\end{tikzpicture}\] 
		Gem 
	\end{minipage}%
	\begin{minipage}{.2\textwidth}
		\centering
		
		\[\begin{tikzpicture}[scale=0.6]

			\vertex (a_1) at (0,-0.5) [fill=black] {};
			\vertex (a_2) at (-1,-0.5) [fill=black] {};
			\vertex (a_3) at (1,-0.5) [fill=black] {};
			\vertex (a_4) at (0,0.5) [fill=black] {};
			\vertex (a_5) at (-1,0.5) [fill=black] {};
			\vertex (a_6) at (1,0.5) [fill=black] {};

			\path
			(a_1) edge (a_2)
			(a_1) edge (a_3)
			(a_1) edge (a_4)
			(a_2) edge (a_5)
			(a_3) edge (a_6)
			(a_4) edge (a_5)
			(a_4) edge (a_6)
			
			;

		\end{tikzpicture}\]
		Domino
	\end{minipage}
	
	\begin{minipage}{.2\textwidth}
		\centering
		
		\[\begin{tikzpicture}[scale=0.6]

			\vertex (a_1) at (0.7,0) [fill=black] {};
			\vertex (a_2) at (1.3,1) [fill=black] {};
			\vertex (a_3) at (0,2) [fill=black] {};
			\vertex (a_4) at (-1.2,1) [fill=black] {};
			\vertex (a_5) at (-0.7,0) [fill=black] {};

			\path
			(a_1) edge (a_2)
			(a_2) edge (a_3)
			(a_3) edge (a_4)
			(a_4) edge (a_5);
			
			\path [dashed]
			(a_1) edge (a_5);

		\end{tikzpicture}\]
		$C_n$,$n \geq 5$
		
	\end{minipage}
	\begin{minipage}{.2\textwidth}
		\centering
		
		\[\begin{tikzpicture}[scale=0.6]

			\vertex (1) at (-1,0) [fill=black] {};
			\vertex (2) at (0,0) [fill=black] {};
			\vertex (3) at (-0.5,1) [fill=black] {};
			\vertex (4) at (-0.5,1.7) [fill=black] {};
			\vertex (5) at (-1.5,-0.5) [fill=black] {};
			\vertex (6) at (0.5,-0.5) [fill=black] {};

			\path 
			(1) edge (2)
			(2) edge (3)
			(3) edge (1)
			(3) edge (4)
			(1) edge (5)
			(2) edge (6);

		\end{tikzpicture}\]
		Net
		
	\end{minipage}
	
	\caption{Forbidden induced subgraphs for distance-hereditary comparability graphs.} \label{fig_8}
\end{figure}
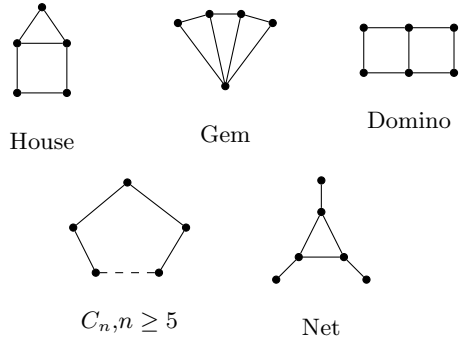

We consolidate the following characterizations of double-arborescences and arborescences in terms of their minimal split-decomposition trees, as a consequence of the work presented so far.

\begin{corollary}\label{ch_double_arb}
	Let $G$ be a graph and $T_{\mathcal{F}}$ be its minimal split-decomposition tree. Then $G$ is a double-arborescence if and only if the following statements hold.
	\begin{enumerate}[label=\rm (\roman*)]
		\item $T_{\mathcal{F}}$ is a clique-star tree.
		\item Each clique component has at most two marked vertices.
		\item There is no marked edge between the centers of two star components. 
		\item $T_{\mathcal{F}}$ does not have any s-leaf-path.
	\end{enumerate}
\end{corollary}

\begin{proof}
	Suppose $G$ is a double-arborescence. Then, by Theorem \ref{P_4-free}, we have $G$ is a $P_4$-free treelike comparability graph. Thus,  we have the statements (i), (ii) and (iii), using Theorem \ref{Treelike_ch}. Further, since $T_{\mathcal{F}}$ is a clique-star tree (statement (i)), it is evident that $G$ is a distance-hereditary graph. Hence, we have $G$ is a $P_4$-free distance-hereditary graph so that the statement (iv) holds by Theorem \ref{lemma_P4}.
	
	Conversely, using the first three statements, we have $G$ is a treelike comparability graph, by Theorem \ref{Treelike_ch}.  Further, in view of Theorem \ref{lemma_P4},  $G$ is $P_4$-free distance-hereditary graph using the statements (i) and (iv). Hence, $G$ is $P_4$-free treelike comparability graph, i.e., $G$ is a double-arborescence by Theorem \ref{P_4-free}. \qed 
\end{proof}

\begin{corollary}\label{main_th}
	Let $G$ be a graph and $T_{\mathcal{F}}$ be its minimal split-decomposition tree. Then $G$ is an arborescence if and only if the following statements hold.
	\begin{enumerate}[label=\rm (\roman*)]
		\item $T_{\mathcal{F}}$ is a clique-star tree.
		\item $T_{\mathcal{F}}$ does not have any center-center path. 
		\item $T_{\mathcal{F}}$ does not have any s-leaf-path.
	\end{enumerate}
\end{corollary}

\begin{proof}
	Suppose $G$ is an arborescence. Then, by Theorem \ref{alter_proof}, we have $G$ is a $(C_4, P_4)$-free graph. Since $G$ is a treelike comparability graph, the statement (i) holds (by Theorem \ref{Treelike_ch}). Further, using the fact that every treelike comparability graph is a distance-hereditary graph, we have $G$  is a $(C_4, P_4)$-free  distance-hereditary graph so that the statements (ii) and (iii) hold by Theorem \ref{lemma_C4} and Theorem \ref{lemma_P4}, respectively. 
	
	For proof of the converse, first note that the statement (i) implies that $G$ is distance-hereditary. Thus, using statements (i) and (ii), $G$ is $C_4$-free (by Theorem \ref{lemma_C4}) and using statements (i) and (iii), we have $G$ is $P_4$-free (by Theorem \ref{lemma_P4}). Hence, $G$ is a $(C_4, P_4)$-free graph so that $G$ is an arborescence (by Theorem \ref{alter_proof}). \qed
\end{proof}

\section{Minimum-Word-Representants} 

In \cite{Bouchet_1}, it was established that the class of distance-hereditary graphs are circle graphs. Thus, being a subclass of distance-hereditary graphs, the  double-arborescences are 2-word-representable. Hence, if $G$ is a double-arborescence on $n$ vertices, the length of its minimum-word-representants $\ell(G) \le 2n$. In this section, we devise an algorithm to find  minimum-word-representants of arborescences and extend it to double-arborescences. Moreover, contributing a class to the open problem in \cite{Eshwar_2024}, we prove that for a double-arborescence $G$ on $n$ vertices with clique number $k$, $\ell(G) = 2n - k$. We refer to \cite{Trotterbook}, for the concepts and notation related to posets used in this section.

\subsection{Word Construction}

Let $G = (V, E)$ be an arborescence on $n$ vertices with clique number $k$. Note that a maximum size clique in a comparability graph can be found in quadratic time \cite{Even_1972}. Suppose $C_{\text{long}}$ is a maximum clique in $G$. As $G$ is a treelike comparability graph, the treelike orientation of $G$ as well as the transitive reduction (hence, the Hasse diagram) can be found in linear time \cite[Theorem 5]{cornelsen2009treelike}. Recall that the Hasse diagram of the induced poset $P_G$ with respect to the arborescence orientation of $G$ is a rooted tree and the root $r$ is the least element or the greatest element of $P_G$. Without loss of generality, we assume that $r$ is the greatest element. Note that the elements of $C_{\text{long}}$ form a longest chain in $P_G$ and $r \in C_{\text{long}}$. In what follows, let $\prec$ be the partial order on the poset $P_G$ and $\prec:$ be the corresponding covering relation, i.e., $a \prec: b$ means $a \prec b$ and there is no element between $a$ and $b$.

In Algorithm \ref{algo-1}, we give a procedure based on breadth-first search (BFS) of $P_G$ to construct a word of minimum length representing $G$. For BFS, one may refer to \cite{cormen}. In what follows, $w$ refers to the output of Algorithm \ref{algo-1}. We have the following remarks on $w$.

\begin{algorithm}[!htb]
	\caption{Constructing minimum-word-representant of an arborescence.}
	\label{algo-1}
	\KwIn{The Hasse diagram $P_G$ of an arborescence $G$ and a longest chain $C_{\text{long}}$ of $P_G$.}
	\KwOut{A word $w$ representing $G$.}
	Initialize $w$ with the root $r$.\\
	Append $r$ in $Q$ (where $Q$ is the queue used in BFS)\\ 	
	\While{$Q$ is not empty}{
		Remove the first element of $Q$, say $a$. \\
		\If{$a$ is not a leaf}{
			Let $a_1, a_2, \ldots, a_t$ be the children of $a$ and append them in $Q$.\\
			\If{$a \in C_{\text{long}}$}{Without loss of generality let $a_1 \in C_{\text{long}}$ and update $w$ by $a_1a_2 \cdots a_twa_ta_{t-1} \cdots a_2$.}
			\Else{Replace the first occurrence of $a$ in $w$ by $a_1a_2 \cdots a_ta$ and the second occurrence of $a$ in $w$ by $a_ta_{t-1} \cdots a_1a$.}
		}
	}
\Return $w$.
\end{algorithm}

\begin{remark}\label{len_w}
	If $G$ is an arborescence on $n$ vertices with clique number $k$, the length of the word $w$ is $2n - k$, as each element of $C_{\text{long}}$ appears exactly once and the remaining elements of $V$ appear twice in $w$. 
\end{remark}

\begin{remark}\label{structure_w}
	The word $w$ is of the form $w_1rw_2$ such that the elements of $C_{\text{long}}$ appear exactly once in $w_1r$ and each element not in $C_{\text{long}}$ has one occurrence  in $w_1$ and the other occurrence in $w_2$.
\end{remark}

We now prove that the word $w$ represents the arborescence $G$ through the following lemmas.

\begin{lemma}\label{cor_1}
	For $a, b \in V$, if $a$ and $b$ are adjacent in $G$, then $a$ and $b$ alternate in $w$.
\end{lemma}

\begin{proof}
	Without loss of generality, suppose $b \prec a$ in $P_G$. Let $b = a_{s} \prec: a_{s-1} \prec: \cdots \prec:a_0 = a$ be the path from $a$ to $b$ in $P_G$. As the root $r$ is the greatest element of $P_G$,  according to the construction of $w$, $a$ is visited first then $a_1, a_2, \ldots, a_{s-1}$ and lastly $b$ is visited.

	\begin{itemize}
		\item Case-1: Suppose $a \notin C_{\text{long}}$. Then  $a_i \notin C_{\text{long}}$ for all $0 \le i \le s$ and consequently all of them occur twice in $w$. For $1 \le t \le s$, using induction on $t$, we show that $w_{\{a, a_t\}} = a_taa_ta$. Hence, in particular for $t = s$, $w_{\{a, b\}} = baba$ so that $a$ and $b$ alternate in $w$.
		
		\qquad For $t = 1$, since $a_1$ is a child of $a$,  as per Step 11 of Algorithm \ref{algo-1}, we have $a_1aa_1a \ll w$ and both $a$ and $a_1$ appear exactly twice in $w$. Hence, $w_{\{a, a_1\}} = a_1aa_1a$. For $t = p$ suppose $w_{\{a, a_{p}\}} = a_{p}aa_{p}a$. For $t = p+1$, since $a_{p+1}$ is a child of $a_p$, both the occurrences of $a_p$ in $w$ are replaced by a word containing $a_{p+1}a_p$ as a subword. Hence, $a_{p+1}a_{p}aa_{p+1}a_{p}a \ll w$ so that $w_{\{a, a_{p+1}\}} = a_{p+1}aa_{p+1}a$.

		\item Case-2: Suppose $a \in C_{\text{long}}$. If $b$ is also in $C_{\text{long}}$, then both $a$ and $b$ occur exactly once in $w$ so that they alternate in $w$. 		
		If $b \notin C_{\text{long}}$, then let $m$ be the smallest possible index with $1 \le m \le s$ such that $a_m \notin C_{\text{long}}$. For $m \le t \le s$, using induction on $t$, we show that $w_{\{a, a_t\}} = a_taa_t$. Hence, for $t = s$, we have $a$ and $b$ alternate in $w$. 
		
		\qquad Since $a_{m-1} \in C_{\text{long}}$, note that $a_{m-1}$ appears exactly once in $w$. If $m = 1$, clearly $a_maa_m \ll w$. For $m > 1$, note that $w_{\{a, a_{m-1}\}} = a_{m-1}a$. Since $a_{m}$ is a child of $a_{m-1}$, in view of Step 7 of Algorithm \ref{algo-1}, $a_m$ appears twice in $w$ such that $a_ma_{m-1}aa_m \ll w$. Thus, for $t = m$, $w_{\{a, a_m\}} = a_maa_m$. For $t = m+p$ suppose $w_{\{a, a_{m+p}\}} = a_{m+p}aa_{m+p}$. For $t = m+p+1$, since $a_{m+p+1}$ is a child of $a_{m+p}$, both the occurrences of $a_{m+p}$ in $w$ are replaced by a word containing $a_{m+p+1}a_{m+p}$ as a subword. Hence, $a_{m+p+1}a_{m+p}aa_{m+p+1}a_{m+p} \ll w$ so that $w_{\{a, a_{m+p+1}\}} = a_{m+p+1}aa_{m+p+1}$. 
	\end{itemize}
	Hence, in any case, $a$ and $b$ alternate in $w$.
 \qed
\end{proof}

\begin{lemma}\label{cor_2}
	For $a, b \in V$, if $a$ and $b$ are not adjacent in $G$, then they do not alternate in $w$.
\end{lemma}

\begin{proof}
  Note that $a$ and $b$ are incomparable elements in $P_G$. As $P_G$ is a rooted tree with $r$ as the greatest element, there exists $c$ in $P_G$ such that $c$ is the least upper bound of $\{a, b\}$.  Let $a = a_s \prec: a_{s-1} \prec: \cdots \prec: a_{1} \prec: c$ be the path from $c$ to $a$, and $b = b_{s'} \prec: b_{s'-1} \prec: \cdots \prec: b_{1} \prec: c$ be the path from $c$ to $b$ in $P_G$. Hence, $c$ is visited first and then $a_i$, $b_j$ are visited in Algorithm \ref{algo-1}. Note that $a_1, b_1$ are children of $c$ and let $a_1$ is visited before $b_1$ in the algorithm. 
	
	\begin{itemize}
		\item Case-1: One of $a$ and $b$ is in $C_{\text{long}}$ but the other is not; say, $a \in C_{\text{long}}$ and $b \notin C_{\text{long}}$. Then $c \in C_{\text{long}}$, $a_i \in C_{\text{long}}$, for all $1 \le i \le s$, and $b_j \notin C_{\text{long}}$, for all $1 \le j \le s'$. For $1 \le t \le s$, using induction on $t$, we first show that $w_{\{a_t, b_1\}} = a_tb_1b_1$.
		
		\qquad For $t = 1$, since $a_1$, $b_1$ are children of $c$ and $a_1 \in C_{\text{long}}$, as per the algorithm, we have $a_1b_1cb_1 \ll w$. Hence, $w_{\{a_1, b_1\}} = a_1b_1b_1$. For $t = p$, suppose $w_{\{a_{p}, b_1\}} = a_{p}b_1b_1$. For $t = p+1$, since $a_{p+1}$ is a child of $a_{p}$ and both $a_{p}$ and $a_{p+1} \in C_{\text{long}}$, as per the algorithm, we have $a_{p+1}a_{p}b_1b_1 \ll w$ so that $w_{\{a_{p+1}, b_1\}} = a_{p+1}b_1b_1$. Hence, in particular for $t = s$, we have $w_{\{a, b_1\}} = ab_1b_1$. 
		
		\qquad Now for $1 \le t \le s'$, using induction on $t$, we show that $w_{\{a, b_t\}} = ab_tb_t$. 		
		For $t = 1$, we have $w_{\{a, b_1\}} = ab_1b_1$. For $t = p$, suppose $w_{\{a, b_{p}\}} = ab_{p}b_{p}$. For $t = p+1$, since $b_{p+1}$ is a child of $b_{p}$, both the occurrences of $b_{p}$ in $w$ are replaced by a word containing $b_{p+1}b_{p}$ as a subword. Thus, we have $ab_{p+1}b_{p}b_{p+1}b_{p} \ll w$ so that $w_{\{a, b_{p+1}\}} = ab_{p+1}b_{p+1}$. Hence, in particular for $t = s'$, we have $w_{\{a, b\}} = abb$ so that $a$ and $b$ do not alternate in $w$. 
		
		\item Case-2: Both $a$ and $b$ are not in $C_{\text{long}}$. Then we have the following cases.
		\item Case-2.1: Suppose $a_1 \notin C_{\text{long}}$ and $b_1 \notin C_{\text{long}}$. Then note that $a_i \notin C_{\text{long}}$ for all $1 \le i \le s$ and $b_j \notin C_{\text{long}}$ for all $1 \le j \le s'$. 
		
		\qquad For $1 \le t \le s$, using induction on $t$, we first show that $w_{\{a_t, b_1\}} = a_tb_1b_1a_t$. 
		Note that $a_1$, $b_1$ are children of $c$ and both $a_1$ and $b_1$ are not in $C_{\text{long}}$. Hence, if $c \in C_{\text{long}}$, as per the algorithm, we have $a_1b_1cb_1a_1 \ll w$ and if $c \notin C_{\text{long}}$, as per the algorithm, we have $a_1b_1cb_1a_1c \ll w$. Hence, in any case, for $t=1$, we have $w_{\{a_1, b_1\}} = a_1b_1b_1a_1$. For $t = p$, suppose $w_{\{a_{p}, b_1\}} = a_{p}b_1b_1a_{p}$. For $t = p+1$, since $a_{p+1}$ is a child of $a_{p}$, both the occurrences of $a_{p}$ in $w$ are replaced by a word containing $a_{p+1}a_{p}$ as a subword. Thus, we have $a_{p+1}a_{p}b_1b_1a_{p+1}a_{p} \ll w$ so that $w_{\{a_{p+1}, b_1\}} = a_{p+1}b_1b_1a_{p+1}$. Hence, in particular for $t = s$, we have $w_{\{a, b_1\}} = ab_1b_1a$.

		\qquad Now for $1 \le t \le s'$, using induction on $t$, we show that $w_{\{a, b_t\}} = ab_tb_ta$. 	For $t = 1$, we have $w_{\{a, b_1\}} = ab_1b_1a$. For $t = p$, suppose $w_{\{a, b_{p}\}} = ab_{p}b_{p}a$. For $t = p+1$, since $b_{p+1}$ is a child of $b_{p}$ and both the occurrences of $b_{p}$ in $w$ are replaced by a word containing $b_{p+1}b_{p}$ as a subword, we have $ab_{p+1}b_{p}b_{p+1}b_{p}a \ll w$ so that $w_{\{a, b_{p+1}\}} = ab_{p+1}b_{p+1}a$. Hence, in particular for $t = s'$, we have $w_{\{a, b\}} = abba$ so that $a$ and $b$ do not alternate in $w$. 
		
		\item Case-2.2: One of $a_1$ and $b_1$ is in $C_{\text{long}}$ but not the other; say, $a_1 \in C_{\text{long}}$ and $b_1 \notin C_{\text{long}}$. Then $c \in C_{\text{long}}$ and $b_j \notin C_{\text{long}}$ for all $1 \le j \le s'$. Let $m (>1)$ be the smallest positive interger such that $a_m \notin C_{\text{long}}$. Therefore $a_i \notin C_{\text{long}}$ for all $m \le i \le s$. 
		
		\qquad  For $m \le t \le s$, using induction on $t$, we show that $w_{\{a_t, b\}} = a_tbba_t$. Hence, in particular for $t=s$, we have $w_{\{a, b\}} = abba$. 		
		As shown in Case-1, we can have $w_{\{a_{m-1}, b\}} = a_{m-1}bb$. For $t = m$, since $a_m$ is a child of $a_{m-1}$ but $a_m \notin C_{\text{long}}$, as per the algorithm, we have $a_ma_{m-1}bba_m \ll w$. Thus $w_{\{a_{m}, b\}} = a_{m}bba_m$.  For $t = m+p$, suppose $w_{\{a_{m+p}, b\}} = a_{m+p}bba_{m+p}$. For $t = m+p+1$, as $a_{m+p+1}$ is a child of $a_{m+p}$ and both the occurrences of $a_{m+p}$ in $w$ are replaced by a word containing $a_{m+p+1}a_{m+p}$ as a subword, we have $a_{m+p+1}a_{m+p}bba_{m+p+1}a_{m+p} \ll w$. Thus $w_{\{a_{m+p+1}, b\}} = a_{m+p+1}bba_{m+p+1}$. Hence in particular for $t = s$, we have $w_{\{a, b\}} = abba$.
\end{itemize}
Hence in any case $a$ and $b$ do not alternate in $w$.
	 \qed
\end{proof}

We now conclude in the following theorem that the word $w$ produced by Algorithm \ref{algo-1} on an arborescence $G$ is a minimum-word-representant of $G$.

\begin{theorem}
	If $G$ is an arborescence on $n$ vertices with clique number $k$, then $\ell(G) = 2n - k$.
\end{theorem}

\begin{proof}
	Let $w$ be the word obtained by Algorithm \ref{algo-1} on an arborescence $G$. From lemmas \ref{cor_1} and \ref{cor_2}, it is evident that $w$ represents $G$.  In view of Remark \ref{len_w}, the length of the word $w$ is $2n - k$ so that $\ell(G) \le 2n - k$. Further, as the size of a maximum clique in $G$ is $k$ we have $2n - k \le \ell(G)$ (cf. \cite[Theorem 2.9]{Eshwar_2024}). Hence, $w$ is a minimum-word-representant of $G$ and $\ell(G) = 2n - k$. \qed 
\end{proof}

We now prove a more general theorem for the minimum-word-representants of  double-arborescences.

\begin{theorem}\label{min_word_doub_arbore}
	Suppose $G$ is a double-arborescence on $n$ vertices with clique number $k$. Then $\ell(G) = 2n - k$.
\end{theorem}

\begin{proof}
	Let $r$ be a root of a double-arborescence $G = (V, E)$ and $P_G$ be the induced poset of $G$ with respect to its double-arborescence orientation. If $C_{\text{long}}$ is a longest chain in $P_G$, then note that the size of $C_{\text{long}}$ is $k$.  
	
    Let $A = \{a \in V \mid a \prec r \ \text{in} \ P_G\}$ and $B = \{a \in V \mid  r \prec a \ \text{in} \ P_G\}$ so that $A \cap B = \{ r \}$ and $A \cup B = V$. Let $P_A$ and $P_B$ be the subposets of $P_G$ on the vertex sets $A$ and $B$, respectively. Then the Hasse diagrams of both $P_A$ and $P_B$ are rooted trees with root $r$. Note that $r$ is the greatest element of $P_A$ and it is the least element of $P_B$. Let $G_A$ and $G_B$ be the induced subgraphs of $G$ on the vertex sets $A$ and $B$, respectively. It can be observed that both $G_A$ and $G_B$ are arborescences.
    
   Let $C^A_{\text{long}} = C_{\text{long}} \cap A$ and $C^B_{\text{long}} = C_{\text{long}} \cap B$ so that $C^A_{\text{long}}$ and $C^B_{\text{long}}$ are longest chains in $P_A$ and $P_B$, respectively. We run Algorithm \ref{algo-1} separately on $P_A$ and $P_B$ by considering the respective longest chains $C^A_{\text{long}}$ and $C^B_{\text{long}}$. Let $u$ and $v$ be minimum-word-representants of $G_A$ and $G_B$, respectively, obtained from Algorithm \ref{algo-1}. Suppose that $u = u_1ru_2$ and $v = v_1rv_2$. Note that the reversal of $u$,  $u^R = u^R_2ru^R_1$ also represents $G_A$. We show that the word $w = u^R_2v_1ru^R_1v_2$ represents the graph $G$.
   
   First note that $w_{A} = u^R_2ru^R_1$. Hence any two vertices $a$ and $a'$ of $G_A$ are adjacent in $G$ if and only if $a$ and $a'$ alternate in $w$. Also $w_B = v_1rv_2$ so that any two vertices $b$ and $b'$ of $G_B$ are adjacent in $G$ if and only if they alternate in $w$. We now show that $r$ alternates with any $a \in V \setminus \{r\}$ in $w$ in the following cases:
   
   \begin{itemize}
   	\item $a \in A \setminus \{r\}$ such that $a \in C^A_{\text{long}}$: Then $a$ occurs exactly once in $u^R$ as well as in $w$. Moreover, since $r$ occurs exactly once in $w$, we have $a$ and $r$ alternate in $w$. Similarly, if $a \in B \setminus \{r\}$ such that $a \in C^B_{\text{long}}$, then both $a$ and $r$ occur exactly once in $w$ so that they alternate in $w$.

   	\item $a \in A \setminus \{r\}$ such that $a \notin C^A_{\text{long}}$: Then $a$ occurs exactly twice in $u^R$ as well as in $w$. Further, in view of Remark \ref{structure_w}, we have $w_{\{a, r\}} = ara$. Similarly, when $a \in B \setminus \{r\}$ but $a \notin C^B_{\text{long}}$, then $w_{\{a, r\}} = ara$. 
    \end{itemize}
    
 Finally, let $a \in A \setminus \{r\}$ and  $b \in B \setminus \{r\}$. In view of Remark \ref{structure_w}, the positions of $a$'s in $u_1$ and $u_2$ and the positions of $b$'s in $v_1$ and $v_2$ are evident as shown in the following cases so that $a$ and $b$ alternate in $w$. 
  \begin{itemize}
  	\item If $a \in C^A_{\text{long}}$ and  $b \in C^B_{\text{long}}$, then $w_{\{a, b\}} = ba$.  
  	\item If $a \in C^A_{\text{long}}$ and $b \notin C^B_{\text{long}}$, then $w_{\{a, b\}} = bab$.
  	\item If $a \notin C^A_{\text{long}}$ and $b \in C^B_{\text{long}}$, then $w_{\{a, b\}} = aba$.
  	\item If $a \notin C^A_{\text{long}}$ and $b \notin C^B_{\text{long}}$, then $w_{\{a, b\}} = abab$.
  \end{itemize}
 Hence, the word $w$ represents the graph $G$. Note that the elements of $C_{\text{long}}$ appear exactly once in $w$ while the rest of the elements of $V$ appear twice in $w$ so that the length of the word $w$ is $2n-k$. Further, since the size of a maximum clique in $G$ is $k$,  we have $\ell(G) = 2n - k$. \qed
\end{proof}

\section{Conclusion}

The characterizations of (double-)arborescences given in this work may be useful in addressing various enumeration related problems of double-arborescences. In \cite{Enumeration_arborescenes}, the arborescences were enumerated both exactly and asymptotically. However, no such results are available for double-arborescences.  Considering the characterizations based on split decomposition for certain subclasses of distance-hereditary graphs, some grammars enumerating the subclasses were produced in \cite{MR3907778}. As our characterizations given in Corollary \ref{ch_double_arb} and Corollary \ref{main_th} for double-arborescences and arborescences are using the split-decomposition trees, taking a clue from the work in \cite{MR3907778}, one may produce grammars to enumerate the (double-)arborescences. Moreover, the recognization algorithm for treelike comparability graphs given in \cite[Theorem 5]{cornelsen2009treelike} may be  customized  to both double-arborescences and arborescences,  using the characterizations given in this work. 

Considering the open problem given in \cite{Eshwar_2024}, we established that the length of a minimum-word-representant of a double-arborescence on $n$ vertices with clique number $k$ is $2n-k$. Based on our observations, we conjecture that the length of minimum-word-representants of a path of double-arborescences is also $2n-k$. Note that a minimum-word-representant of a double-arborescence or an arborescence need not be unique. In fact, the minimum-word-representant produced by Algorithm \ref{algo-1} on a double-arborescence may differ depending on the maximum size clique chosen. It is also interesting to enumerate all minimum-word-representants of double-arborescences and arborescences, in line with the work in \cite{Marisa_2020}.

\end{document}